\documentclass[reqno,11pt]{amsart}
 
\usepackage{amsmath}
\usepackage{amsthm}  
\usepackage{verbatim}
\usepackage{enumerate} 
\usepackage{mathtools}  
\usepackage{amssymb}
\usepackage{mathrsfs}
\usepackage[top=1.5in, bottom=1.5in, left=0.7in, right=0.7in]{geometry}  
\usepackage[backref=page, colorlinks]{hyperref}
\hypersetup{
	colorlinks,
	citecolor=blue,
	linkcolor=red
}   
\numberwithin{equation}{section}
\usepackage{xypic}
\usepackage{bm}
\usepackage{tikz}
\usetikzlibrary{decorations.pathreplacing}
\usepackage{xcolor}
\usepackage{float}
\usepackage{cleveref}
\usepackage{comment}
\usepackage[abs]{overpic}

\newtheorem{theorem}{\textbf{Theorem}}[section]
\newtheorem{theorem*}{\textbf{Theorem}}

\newtheorem{proposition}[theorem]{\textbf{Proposition}}
\newtheorem{lemma}[theorem]{\textbf{Lemma}}

\newtheorem{claim}[theorem]{\textbf{Claim}}

\newtheorem{example}[theorem]{\textbf{Example}}
\newtheorem{conjecture}[theorem]{\textbf{Conjecture}}
\newtheorem{definition/proposition}[theorem]{\textbf{Definition/Proposition}}

\def\N{{\mathbb N}}
\def\R{\mathbb{R}}
\def\Z{{\mathbb Z}}

\def\D{{\mathbb D}}

\def\Q{{\mathbb Q}}

\def\cA{{\mathcal A}}

\def\cI{{\mathcal I}}

\def\rd{{\rm d}}

\def\la{\langle\,}
\def\ra{\,\rangle}

\DeclareMathOperator{\cw}{cw}
\DeclareMathOperator{\wl}{wl}

\newcommand{\Addresses}{{
		\bigskip
		\footnotesize

        Youlin Li, \par\nopagebreak
	    \textsc{School of Mathematical Sciences, Shanghai Jiao Tong University, China}\par\nopagebreak
		\textit{E-mail address}: \href{mailto:liyoulin@sjtu.edu.cn}{liyoulin@sjtu.edu.cn}

        \bigskip

	    Zhengyi Zhou, \par\nopagebreak
	    \textsc{Morningside Center of Mathematics, Chinese Academy of Sciences;}\par\nopagebreak
         \textsc{Academy of Mathematics and Systems Science, Chinese Academy of Sciences, China}\par\nopagebreak
		\textit{E-mail address}: \href{mailto:zhyzhou@amss.ac.cn}{zhyzhou@amss.ac.cn}

}}
\title{Algebraically overtwisted tight $3$-manifolds from $+1$ surgeries}
\author{Youlin Li and Zhengyi Zhou}
\begin{document}
	\maketitle
\begin{abstract}
We execute Avdek's algorithm to find many algebraically overtwisted and tight $3$-manifolds by contact $+1$ surgeries. In particular, we show that a contact $1/k$ surgery on the standard contact 3-sphere  along any Legendrian positive torus knot with the maximal Thurston–Bennequin invariant yields an algebraically overtwisted and tight $3$-manifold, where $k$ is a positive integer. 
\medskip

\end{abstract}
\section{Introduction}
It is a fundamental question to understand the boundary between flexibility and rigidity phenomena in symplectic and contact topology. An example in the context of contact structures is whether overtwisted contact structures can be characterized using holomorphic curves. One natural candidate is the contact Ozsv\'ath-Szab\'o invariant in dimension $3$ \cite{OS}, as its vanishing is a necessary condition for overtwistedness. However, it is not a sufficient condition \cite{GHv}. Another natural candidate, which works for any dimension, is the vanishing of the contact homology, as Bourgeois and van Koert \cite{BvK} showed that contact homology vanishes for any overtwisted contact manifold. Therefore Bourgeois and Niederkr{\"u}ger \cite{AOT} introduced the concept of algebraically overtwisted manifolds to mean those contact manifolds with vanishing contact homology. The insufficiency of algebraic overtwistedness to determine tightness was obtained quite recently by Avdek \cite{avdek2020combinatorial} in dimension $3$ by showing that contact $1/k$ ($k\in\N_+$, i.e.\ $k$ is a positive integer) surgery along a Legendrian right-handed trefoil with the maximal Thurston–Bennequin invariant is algebraically overtwisted, yet still tight by \cite{LS}. In this note, we execute Avdek's algorithm to find more algebraically overtwisted but tight $3$-manifolds from contact $+1$ surgeries. Let $\Lambda$ be a Legendrian knot in $(S^{3}, \xi_{std})$,  we denote the contact 3-manifold obtained by contact $1/k$ surgery along $\Lambda$  by $(S^{3}_{1/k}(\Lambda), \xi_{1/k}(\Lambda))$. We first consider the case where the Legendrian knots are positive torus knots.

\begin{theorem}\label{thm:main}
    Let $\Lambda$ be a non-trivial Legendrian positive torus knot with  maximal Thurston–Bennequin invariant in $(S^{3}, \xi_{std})$. Then $(S^{3}_{1/k}(\Lambda), \xi_{1/k}(\Lambda))$ is algebraically overtwisted and tight for $k\in \N_+$.
\end{theorem}
We also consider the case where the Legendrian knots are rainbow closures of some positive braids, see Figure~\ref{fig:twistedtorus_intro} for an example. The notion of Legendrian rainbow closures of positive braids was introduced in \cite[Section 6.5]{STZ} using front diagrams. According to \cite[Theorem 2]{T}, a Legendrian rainbow closure of a positive braid attains the maximal Thurston-Bennequin invariant of the underlying topological knot type.

\begin{theorem}\label{thm:2}
    Let $\sigma_1$, $\sigma_2$,  $\cdots$, $\sigma_{p-1}$ be the generators of the $p$-strand braid group, and let $\Lambda$ be a Legendrian knot that is the Legendrian rainbow closure of a positive braid $$(\sigma_{1}\cdots\sigma_{p-2}\sigma_{p-1})^{q_1}(\sigma_{p-p_{2}+1}\cdots\sigma_{p-2}\sigma_{p-1})^{q_2}\cdots (\sigma_{p-p_{N}+1}\cdots\sigma_{p-2}\sigma_{p-1})^{q_N}, q_i>0.$$
    Then $(S^{3}_{1/k}(\Lambda), \xi_{1/k}(\Lambda))$ is algebraically overtwisted and  tight for $k\in \N_+$ when $q_1\gg 0$. When $N=2$ and $p_2=p-1$, $(S^{3}_{1/k}(\Lambda), \xi_{1/k}(\Lambda))$ is algebraically overtwisted and  tight for $k\in \N_+$ if $q_1>1$.
\end{theorem}
\begin{figure}[htb] {\small
\begin{overpic}[scale=0.5]
{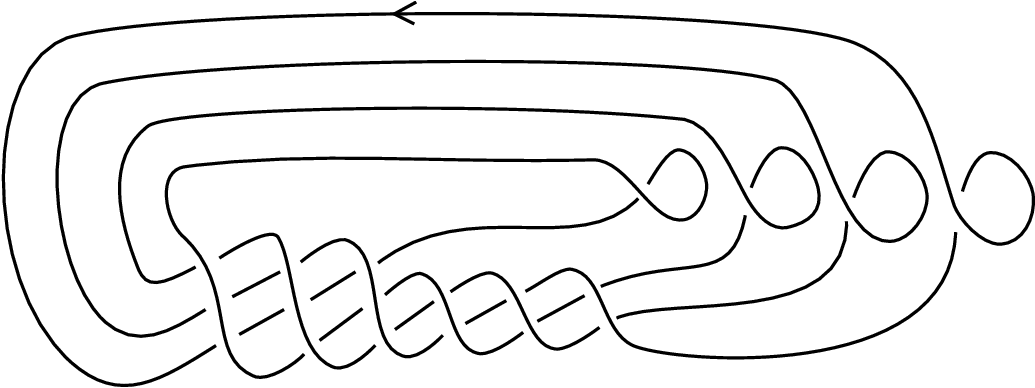}
\end{overpic}}
\caption{A Legendrian knot which is the rainbow closure of the braid $(\sigma_1\sigma_2\sigma_3)^{3}(\sigma_2\sigma_3)^{3}.$  }
\label{fig:twistedtorus_intro}
\end{figure}

All the knots discussed above, including the positive torus knots, are Legendrian rainbow closures of positive braids, which attain the maximal Thurston–Bennequin invariants. Based on this observation, we propose the following conjecture.
\begin{conjecture}
    Let $\Lambda$ be a non-trivial Legendrian knot which is the Legendrian rainbow closure of a positive braid. Then $(S^{3}_{1/k}(\Lambda), \xi_{1/k}(\Lambda))$ is algebraically overtwisted and tight for $k\in \N_+$.
\end{conjecture}
The tightness follows from \cite{LS}, see \Cref{lem:tight}. To show the vanishing of contact homology after a contact $1/k$ surgery for $k\in \N_+$, Avdek's algorithm reduces the problem into showing a set of linear inequalities can only have the trivial solution. While this approach often leads to an algebraically overtwisted contact manifold, applying this algorithm to general knots can be challenging. Additionally,  we show by explicit computations that contact $1/k$ surgeries along Chekanov's  Legendrian $5_2$ knots also result in algebraically overtwisted and tight 3-manifolds.

\subsection*{Acknowledgments} We thank Russell Avdek for helpful comments on a preliminary version of the paper and the anonymous referee for many helpful suggestions. Y.L. is supported by the National Natural Science Foundation of China under Grant No.\ 12271349.
Z.Z. is supported by the National Key R\&D Program of China under Grant No.\ 2023YFA1010500, the National Natural Science Foundation of China under Grant No.\ 12288201 and 12231010.
\section{Avdek's algorithm}
In \cite{avdek2020combinatorial}, Avdek gave descriptions of homology classes, Conley-Zehnder indices, and, notably, the intersection grading of Reeb orbits derived from the chord-to-orbit correspondence in contact $\pm1$ surgeries. The proof of \cite[Theorem 1.2]{avdek2020combinatorial} used all of this information. A key simplification used in \cite[Theorem 1.2]{avdek2020combinatorial}
is that the Conley-Zehnder index in the chord-to-orbit correspondence is bounded below by the word length, which enables us to focus solely on orbits  from single chords. This property still holds for rainbow closures of positive braids, and \Cref{thm:main,thm:2} boil down to showing a system of inequalities does not have non-trivial solutions from the intersection grading only, while the additional consideration of homology in \cite[\S 12.5.3]{avdek2020combinatorial} is not necessary. In this section, we revisit Avdek's chord-to-orbit correspondence and highlight its pertinent properties for our discussion.

\subsection{The Chord-to-orbit correspondence}

Let $(\R^3_{\Lambda^+},\xi_{\Lambda^+})$ ($(S^3_{\Lambda^+},\xi_{\Lambda^+})$, resp.) be the  contact 3-manifold obtained by applying contact $+1$ surgery on $(\R^3, \xi_{std})$ ($(S^3, \xi_{std})$, resp.) along a Legendrian knot  $\Lambda$.   
\begin{theorem}[{\cite[Theorem 5.1 (1)]{avdek2020combinatorial}}]
    There exists a contact form $\alpha_{\epsilon}$ on $(\R^3_{\Lambda^+},\xi_{\Lambda^+})$, such that there exists a one-to-one correspondence\footnote{Strictly speaking, this correspondence works for orbits with an action upper bound, which will go to infinity if $\epsilon\to 0$. In practice, e.g.\ the results in this paper, it suffices to work with a fixed $\epsilon\ll1$.} between closed orbits of $R_{\epsilon}$ (Reeb vector field for $\alpha_{\epsilon}$) with cyclic words of the Reeb chords of $\Lambda\subset (\R^3,\xi_{std})$. Here the Reeb chords are computed using the standard contact form $\rd z - y\rd x$.
\end{theorem}
For chords $c_i$, we use $(c_1\ldots c_n)$ to denote the Reeb orbit corresponding to the cyclic word $c_1\ldots c_n$. On the other hand, given a Reeb orbit of $R_{\epsilon}$, we use $\cw(\gamma)$ to denote the corresponding cyclic word of Reeb chords. By $\wl(\gamma)$, we mean the word length of $\cw(\gamma)$. For a Reeb orbit $\gamma$ of $R_{\epsilon}$ and Reeb chords $c_i$ of $\Lambda$, we define contact actions
$$\cA(\gamma):=\int \gamma^*\alpha_{\epsilon}, \quad \cA(c_1\ldots c_n):=\sum_{i=1}^n \int c_i^*(\rd z - y\rd x).$$
\begin{proposition}[{\cite[Proposition 5.13]{avdek2020combinatorial}}]\label{prop:period}
For any Reeb orbit $\gamma$ of $R_{\epsilon}$ we have 
$$|\cA(\gamma)-\cA(\cw(\gamma))|<3\epsilon \wl(\gamma).$$
\end{proposition}

If we orient the Legendrian knot $\Lambda$ and its meridian $\mu$ as in \Cref{fig:meridian}, then $\mu$ is a generator of $H_1(\R^3_{\Lambda^+};\Z)$ and is subject to the relation $(tb(\Lambda)+1)\mu=0$.

\begin{figure}[htb] {\small
\begin{overpic}
{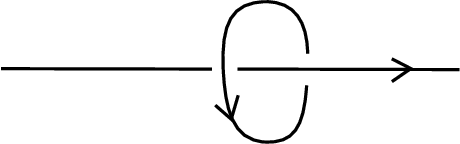}
\put (100,10) {$\mu$}
\put (200,40) {$\Lambda$}
\end{overpic}}
\caption{Default orientations of meridians  }
\label{fig:meridian}
\end{figure}

Let $c_1\ldots c_n$ be a cyclic word of Reeb chords. We define the push-out loop $P(c_1\ldots c_n)$ in the complement $\R^3\backslash N_{\epsilon}(\Lambda)$, which is homotopic to the orbit $(c_1\ldots c_n)$. Here $N_{\epsilon}(\Lambda)$ is a tubular neighborhood of $\Lambda$, where the surgery applies. The push-out loop $P(c_1\ldots c_n)$ in the Lagrangian projection can be described as follows: Starting at the end point of the chord $c_1$, and following the direction of the Legendrian knot, we push the Legendrian arc to the left until we reach the starting point of $c_2$. From there, we trace the Reeb chord to the end point of $c_2$ and repeat the procedure until it is closed, see \Cref{fig:push-outs }. In the terminology of \cite[\S 9.4]{avdek2020combinatorial}, our push-outs are the push-outs using positive capping paths $\eta_{j_k,j_{k+1}}$ for the $+1$ surgery (ref. \cite[Figure 19]{avdek2020combinatorial}).

\begin{figure}[htb] {\small
\begin{overpic}
{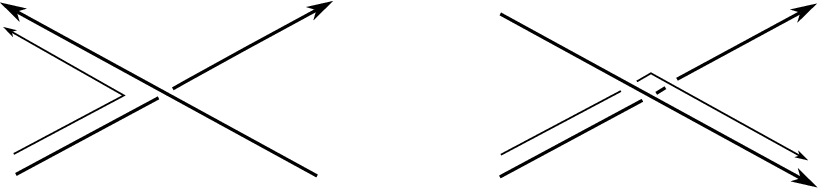}

\end{overpic}}
\caption{Push-outs of the Reeb orbits.  }
\label{fig:push-outs }
\end{figure}

\subsection{Intersection grading and holomorphic curves}
The Lagrangian projection of a Legendrian knot divides $\R^2$ into several bounded regions and an unbounded region. We consider an $\epsilon$ neighborhood $N_{\epsilon}$ of $\Lambda$, such that the complement of the Lagrangian projection of $N_{\epsilon}$ has the same number of components as the complement of the Lagrangian projection of $\Lambda$. We order the components of the complement of the Lagrangian projection of $N_{\epsilon}$ as $R_0, R_1, \ldots, R_K$, with $R_0$ representing the unbounded region. In \cite[\S  11 and \S 12.4]{avdek2020combinatorial}, Avdek introduced an intersection index $\cI_{\Lambda}(\sum_{j=1}^m \gamma_j)$, where $\gamma_1,\ldots,\gamma_p$ are Reeb orbits contained in $\R^3_{\Lambda^+}\backslash (\cup_{i=0}^K \R_z\times R_i)$ and $\sum_{j=1}^m \gamma_j$ is null-homologous in $\R^3_{\Lambda^+}$, as follows. Let $(x_i,y_i)\in R_i$ for $i>0$. Then the intersection grading in $\Z\la R_i\ra$ of $\sum_{j=1}^m \gamma_j$ is given by 
$$\cI_{\Lambda}(\sum_{j=1}^m \gamma_j) = \sum_{i=1}^K ((x_i,y_i)\times \R_z\cdot D)R_i\in \Z\la R_1,\ldots,R_K\ra,$$
where $D$ is a singular chain in $\R^3_{\Lambda^+}$ such that $\partial D = \sum_{j=1}^m \gamma_j$, and $(x_i,y_i)\times \R_z \cdot D$ is the intersection number of $(x_i,y_i)\times \R_z$ and $D$ which is independent of $(x_i,y_i)$ and the relative homology class of $D$. To get rid of the dependence of $D$, we assume $tb(\Lambda)\ne -1$, which is the case for all knots considered in this paper. Then $H_2(\R^3_{\Lambda_+};\Z)$ is torsion and hence the intersection number is independent of $D$. The intersection grading of $\sum_{j=1}^m \gamma_j$ remains unchanged when we vary it within a cobordism class in $\R^3_{\Lambda^+}\backslash (\cup_{i=0}^K \R_z\times R_i)$. The intersection grading is clearly linear with respect to the union of loop collections with trivial homology classes, i.e.
 $$\cI_{\Lambda}(\sum_j \gamma_j+\sum_{j'} \delta_{j'})=\cI_{\Lambda}(\sum_j \gamma_j)+\cI_{\Lambda}(\sum_{j'} \delta_{j'}).$$
In the computation of contact homology, we will choose an almost complex structure such that $\R_s\times (x_i,y_i)\times \R_z \subset \R_{s}\times \R^3_{\Lambda^+}$ is holomorphic. As a consequence, we have the following constraint from positivity of intersection, which makes the intersection grading a useful tool for excluding certain holomorphic curves in the contact homology computation. 

Here we fix our notion for contact homology for both closed contact 3-manifolds and punctured contact 3-manifolds. For a closed Reeb orbit $\gamma$, we denote the associated formal variable as  $q_{\gamma}$. The differential in contact homology is denoted by $\partial_{CH}$. 
 
 \begin{proposition}[{\cite[Discussions before \S 12.5]{avdek2020combinatorial}}]\label{prop:no_curve}
     If there is a contact homology differential from $q_{\gamma_0}$ to $q_{\gamma_1}\ldots q_{\gamma_m}$, then 
     $$\cI_{\Lambda}(\gamma_0-\sum_{j=1}^m \gamma_j)\geq0,$$
     i.e.\ all coefficients of $\cI_{\Lambda}$ is non-negative. The intersection grading is defined, as $\gamma_0-\sum_{j=1}^m \gamma_j$ is null-homologous by the existence of contact homology differential.
 \end{proposition}
 
 Let $\mu$ and $\lambda$ be the meridian and longitude of $\Lambda$ following the orientation convention in \Cref{fig:meridian}, where $\lambda$ is induced by a Seifert surface. Then
\begin{equation}\label{eqn:wind}
    \cI_{\Lambda}(\lambda)=\sum_i wind(\Lambda,R_i)R_i,
\end{equation}
where $wind(\Lambda,R_i)$ is the winding number of $\Lambda$ around the bounded region $R_i$. Since $\lambda+(tb(\Lambda)+1)\mu$ is null-homologous in $\R^3_{\Lambda^+}\backslash (\cup_{i=0}^K \R_z\times R_i)$, we have 
\begin{equation}\label{eqn:mu}
    \cI_{\Lambda}((tb(\Lambda)+1)\mu)=-\cI_{\Lambda}(\lambda)=-\sum_i wind(\Lambda,R_i)R_i.
\end{equation}

\subsection{Rational intersection grading}
We assume $tb(\Lambda)\ne -1$. Then every Reeb orbit in $(\R^3_{\Lambda^+}, \xi_{\Lambda^+})$ represents a torsion homology class. So for any Reeb orbit $\gamma$ (which is contained in  $\R^3_{\Lambda^+}\backslash (\cup_{i=0}^K \R_z\times R_i)$), there exists a positive integer $d$ such that $d\gamma$ is null-homologous. We define
$$\cI_{\Lambda}(\gamma):=\frac{1}{d}\cI_{\Lambda}(d\gamma)\in \Q\la R_i\ra.$$
It is straightforward to verify that $\cI_{\Lambda}(\gamma)$ is independent of  $d$. This rational intersection grading is linear, i.e., 
$$\cI_{\Lambda}(\gamma+\delta) = \cI_{\Lambda}(\gamma)+\cI_{\Lambda}(\delta), \quad \cI_{\Lambda}(-\gamma)=-\cI_{\Lambda}(\gamma).$$
The meridian $\mu$ is also contained in  $\R^3_{\Lambda^+}\backslash (\cup_{i=0}^K \R_z\times R_i)$, and from \eqref{eqn:mu}, we have 
$$\cI_{\Lambda}(\mu)=\frac{-1}{tb(\Lambda)+1}\sum_i wind(\Lambda,R_i)R_i.$$
Let $c$ be a Reeb chord, we define 
$$\cI_{\Lambda}(c):=\cI_{\Lambda}((c))=\cI_{\Lambda}(P(c)).$$

\begin{proposition}
    $\cI_{\Lambda}(c) = \sum_i wind(P(c),R_i) R_i+lk(P(c),\Lambda)\cI_{\Lambda}(\mu).$
\end{proposition}
\begin{proof}
We move $P(c)$ vertically upwards in the $z$-direction to $P'(c)$ so that $P'(c)$ stays above $\Lambda$.  We see that $P(c)-P'(c)$ is homologous to $lk(P(c),\Lambda)\mu$ in  $\R^3_{\Lambda^+}\backslash (\cup_{i=0}^K \R_z\times R_i)$. Then the claim follows from  $\cI_{\Lambda}(P'(c))=\sum_i wind(P(c),R_i) R_i$.
\end{proof}
In general, let $c_1\ldots c_n$ be a cyclic word of Reeb chords, we write 
$$\cI_{\Lambda}(c_1\ldots c_n):=\cI_{\Lambda}((c_1\ldots c_n))=\cI_{\Lambda}(P(c_1\ldots c_n)).$$
From \Cref{ex:non-add}, we see that $\cI_{\Lambda}$ is not additive with respect to the concatenation of words, i.e.\ $\cI_{\Lambda}(w_1w_2)\ne \cI_{\Lambda}(w_1)+\cI_{\Lambda}(w_2)$ in general. However, for rainbow closures of positive braids, the following proposition allows us to consider only the intersection gradings of type $\cI_{\Lambda}(c)$, where $c$ is a single chord. 

\begin{proposition}\label{prop:positive}
Let $\Lambda$ be the Legendrian rainbow closure of a positive braid, see the left of Figure~\ref{fig:rot} for its Lagrangian projection. Then the Conley-Zehnder index of $(c_1\ldots c_n)$ is bounded below by $n$. In particular, the SFT degree of an orbit $\gamma$ is at least $\wl(\gamma)-1$.
\end{proposition}
\begin{proof}
Assuming each crossing of the Lagrangian projection of $\Lambda$ is in good position (ref.\ \cite[\S 4.2]{avdek2020combinatorial}), we will show that the rotation angle $\theta_{i,j}$ (ref.\ \cite[\S 3.3.1]{avdek2020combinatorial}) of any pair $(c_i, c_j)$ is either $\frac{\pi}{2}$ or $\frac{3\pi}{2}$. The capping path (ref.\ \cite[\S 3.3]{avdek2020combinatorial}) of the pair $(c_i, c_j)$ is denoted by $\eta_{i,j}$. Since $\Lambda$ is a positive braid closure with Lagrangian projection from the braid representation, the capping path $\eta_{i,j}$ consists of three types of arcs in Figure~\ref{fig:rot}. The Type 1 arc contributes $\frac{\pi}{2}$ to the rotation angle. The Type 2 arc contributes $-\frac{3\pi}{2}$ to the rotation angle. The Type 3 arc contributes $\frac{3\pi}{2}$ to the rotation angle. If $c_i$ belongs to the braid area, then $\eta_{i,j}$ consists of one Type 1 arc, $m$ Type 2 arcs, and $m$ Type 3 arcs, for some integer $m$. So the rotation angle $\theta_{i,j}$ of  $(c_i, c_j)$ is $\frac{\pi}{2}$.  If $c_i$ does not belong to the braid area, then $\eta_{i,j}$ consists of $m+1$ Type 3 arcs and $m$ Type 2 arcs, for some integer $m$. So the rotation angle $\theta_{i,j}$ of $(c_i, c_j)$ is $\frac{3\pi}{2}$. 

Therefore the rotation number $rot_{i, j}=\lfloor\frac{\theta_{i,j}}{\pi}\rfloor$ (ref.\ \cite[\S 3.3.1]{avdek2020combinatorial})  of the pair $(c_i, c_j)$ is either $0$ or $1$. According to \cite[Theorem 7.1]{avdek2020combinatorial}, the Conley-Zehnder index of $(c_1\ldots c_n)$ is $\sum_{k=1}^{n}(rot_{k, k+1}+1)\geq n$.
\end{proof}

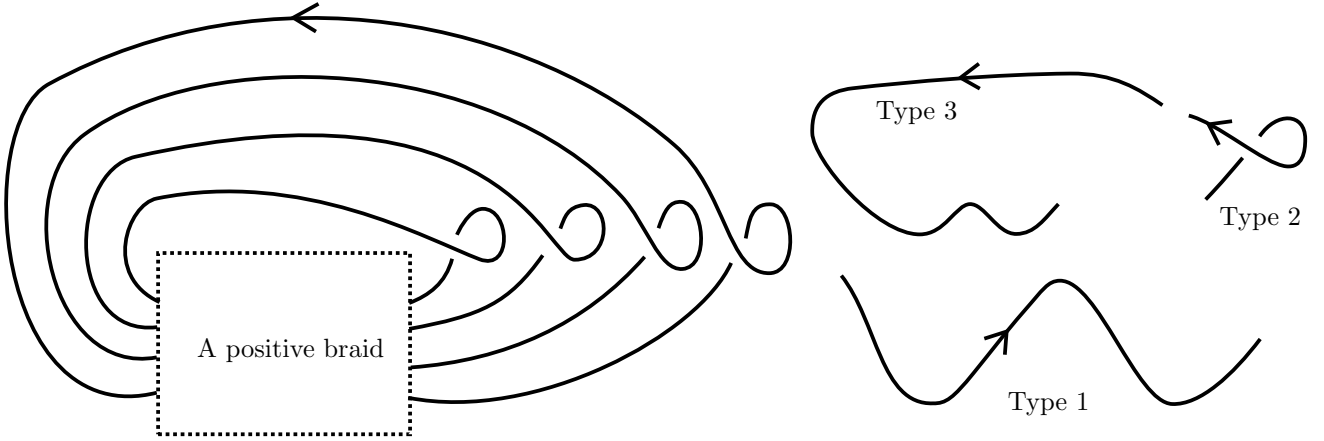
\begin{figure}[htb] {\small
\begin{overpic}
{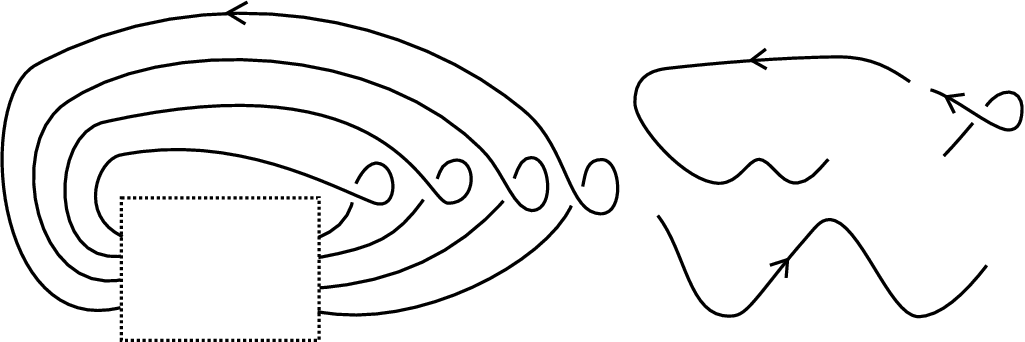}
\put (380,10) {Type $1$}
\put (460,80) {Type $2$}
\put (330,120) {Type $3$}
\put (70,30) { A positive braid}
\end{overpic}}
\caption{The left part is the Legendrian rainbow closure of a positive braid. The right part contains three types of arcs in a capping path.}
\label{fig:rot}
\end{figure}

\subsection{Source of holomorphic planes}
Let $u$ be an embedded RSFT disk $u:\D \backslash\{p_k\}\to \R^2$ for some positive boundary punctures $\{p_k\}$ that completely covers a component of $\R^2\backslash \pi_{x,y}(\Lambda)$ as in \Cref{fig:disk}. Let $c_1,\ldots,c_n$ be the cyclic chords of the disk. These disks give rise to rigid (modulo translation) holomorphic disks in the symplectization of $\R^3$ with boundary condition on the $\R\times \Lambda$ and positive asymptotic conditions given by $\{c_i\}$ by \cite[Corollary 11.2]{avdek2020combinatorial}. Such disks contribute to Ng's Rational symplectic field theory for Legendrian knots \cite{zbMATH05822877}.

\begin{figure}[htb] {\small
\begin{overpic}
{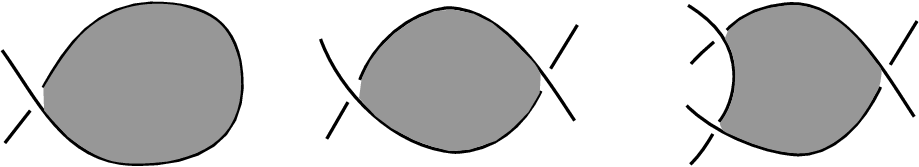}

\end{overpic}}
\caption{RSFT disks with only positive punctures.}
\label{fig:disk}
\end{figure}

\begin{theorem}[{\cite[Theorem 12.2]{avdek2020combinatorial}}]\label{thm:curve}
    Given an embedded RSFT disk as above, the constant term of $\partial_{CH}(q_{(c_1\ldots c_n)})$ is $\pm 1$.
\end{theorem}
\section{Torus knots}\label{s3}

We consider a Legendrian $(p,q)$ torus knot $\Lambda$ with maximal Thurston–Bennequin invariant $pq-p-q$, where $p, q$ are coprime positive integers, both greater than $1$. By the classification of Legendrian torus knots in \cite{EH}, there is a unique Legendrian $(p,q)$ torus knot whose Thurston–Bennequin invariant is $pq-p-q$. Without loss of generality, we assume that $p<q$. The Legendrian $(p,q)$ torus knot $\Lambda$ is presented by its Lagrangian projection. See Figure~\ref{figure:Torus knot } for an example. In the lower part of the Lagrangian projection, there is a braid area. We label the regions in the braid area by $R_{i,j}$ for $1\le i\le p-1$ and $1\le j \le q-1$.  The remaining regions are labeled by $A_1,B_1,\ldots,A_p,B_p$, where $B_i$ is the small disk next to $\alpha_i$ and $A_i$ is the region on the other side. The vertices of $R_{i,j}$ are $r_{i,j},r_{i+1,j},r_{i,j+1},r_{i+1,j+1}$, if one of the indices exceed the bounds then $R_{i,j}$ degenerates to a triangle. We orient the knot such that in the braid area it is going from left to  right.

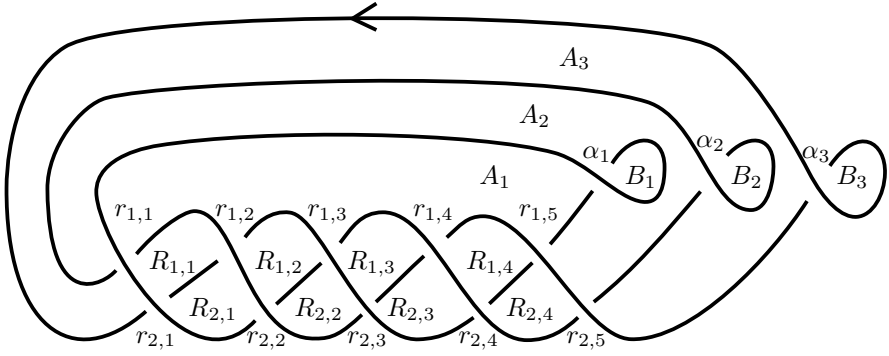
\begin{figure}[htb] {\small
\begin{overpic}
{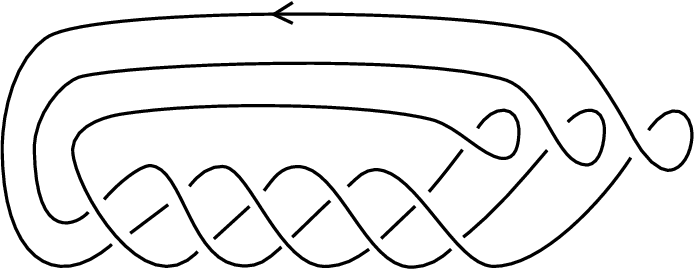}
\put(180, 60){$A_1$}
\put(195, 83){$A_2$}
\put(210, 105){$A_3$}

\put(235, 60){$B_1$}
\put(275, 60){$B_2$}
\put(315, 60){$B_3$}

\put(219, 70){$\alpha_1$}
\put(262, 74){$\alpha_2$}
\put(302, 70){$\alpha_3$}

\put(42, 48){$r_{1,1}$}
\put(80, 48){$r_{1,2}$}
\put(115, 48){$r_{1,3}$}
\put(155, 48){$r_{1,4}$}
\put(195, 48){$r_{1,5}$}

\put(50, 0){$r_{2,1}$}
\put(92, 0){$r_{2,2}$}
\put(130, 0){$r_{2,3}$}
\put(172, 0){$r_{2,4}$}
\put(213, 0){$r_{2,5}$}

\put(55, 28){$R_{1,1}$}
\put(70, 10){$R_{2,1}$}
\put(95, 28){$R_{1,2}$}
\put(110, 10){$R_{2,2}$}
\put(130, 28){$R_{1,3}$}
\put(145, 10){$R_{2,3}$}
\put(175, 28){$R_{1,4}$}
\put(190, 10){$R_{2,4}$}

\end{overpic}}
\caption{Markings of regions and crossings.}
\label{figure:Torus knot  }
\end{figure}

From \eqref{eqn:mu}, we can compute
\begin{equation}\label{eqn:mu_torus}
\cI_{\Lambda}(\mu) = \frac{-1}{(p-1)(q-1)}\left(\sum_{j=1}^{q-1}\sum_{i=1}^{p-1} (p-i)R_{i,j}+\sum_{i=1}^p ((p+1-i)A_i+(p-1-i)B_i)\right).
\end{equation}
Now note that we can present the knot, such that the area of $B_1$ is much smaller than the areas of $B_2$, $B_3$, $\cdots$, $B_p$. Thus in view of the period estimate in \Cref{prop:period}, $\partial_{CH}(q_{(\alpha_1)})$ only involves Reeb orbits from the $r_{i,j}$ chords. By \Cref{thm:curve},  the constant term of $\partial_{CH}(q_{(\alpha_1)})$ is $\pm 1$. Moreover, by \Cref{prop:positive}, the subalgebra with SFT degree $0$ is generated by orbits from single chords. So the non-constant term of $\partial_{CH}(q_{(\alpha_1)})$ is generated by orbits coming from single chords. Therefore, to prove $(\R^3_{\Lambda^+}, \xi_{\Lambda^+})$ has vanishing contact homology, it suffices to prove the following claim:
\begin{claim}\label{claim} The inequality 
\begin{equation}\label{eqn:inequality}
\cI_{\Lambda}(\alpha_1)-\sum_{1\le i \le p-1, 1\le j \le q} x_{i,j}\cI_{\Lambda}(r_{i,j})\ge 0, \quad x_{i,j}\ge 0
\end{equation}
has only the trivial solution $x_{i,j}=0$ for all $1\le i \le p-1$ and $1\le j \le q$.
\end{claim}
The claim implies that $\partial_{CH}(q_{(\alpha_1)})$ contains no terms involving $\prod_{i,j}(r_{i,j})^{x_{i,j}}$ if any of $x_{i,j}$ are positive.

\begin{proposition}
    $\cI_{\Lambda}(\alpha_1)=B_1$.
\end{proposition}
\begin{proof}
    In the terminology of \cite[\S 9.4]{avdek2020combinatorial}, we use the negative push-out $\overline{\eta}$ as it traverses a smaller portion of the knot. Since the choice of capping path does not affect the intersection grading, the claim follows from \Cref{fig:alpha}.
    \end{proof}

\begin{figure}[htb] {\small
\begin{overpic}
{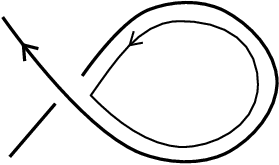}
\put(-8, 30){$\cdots$}
\put(25, 45){$\alpha_1$}
\put(82, 35){$B_1$}
    \end{overpic}}
    \caption{Push-out of $(\alpha_1)$ using the negative capping path.}
    \label{fig:alpha}
    \end{figure}

We use $P_{i,j}$ to denote the push-out of the Reeb orbit corresponding to $r_{i,j}$. By a  strand, we mean a piece of $P_{a,b}$ starts at the leftmost end of the braid area and goes around once back to the head of the braid area.
\begin{lemma}\label{lemma:sum}
    For all $1\le a \le p-1, 1\le b \le q$, we have
    $$\sum_{j=1}^{q-1}\cI_{\Lambda}(r_{a,b})_{R_{1,j}}\ge 0.$$
    Moreover, the above quantity is positive if and only if $P_{a,b}$ contains the most interior long overhead arc as shown in Figure~\ref{figure:Torus knot 1 }.
\end{lemma}
\begin{proof}
    We consider a strand in $P_{a,b}$. Each step-up (see \Cref{figure:Torus knot 1 }) of the strand, including the crossing at $r_{a,b}$, contributes a $\mu$ from linking. This contribution translates to $\frac{-1}{(p-1)(q-1)}((q-1)(p-1))=-1$ for $\sum_{j=1}^{q-1}\cI_{\Lambda}(r_{a,b})_{R_{1,j}}$. Each crossing at $\alpha_i$, contributes $-\mu$, i.e.\ $1$ to $\sum_{j=1}^{q-1}\cI_{\Lambda}(r_{a,b})_{R_{1,j}}$. Furthermore, each horizontal piece (see \Cref{figure:Torus knot 1 }) of the strand in the braid area contributes $1$ to $\sum_{j=1}^{q-1}\cI_{\Lambda}(r_{a,b})_{R_{1,j}}$ if it is followed by a step-up (in the direction of the knot, i.e.\ from left to right in region with step-ups).  Consequently, the total contribution is at least $0$, as each horizontal piece that contributes is cancelled with the step-up followed, and a positive contribution from $\alpha_i$ will cancel with the negative contribution from the step-up from the leftmost descending arc in the braid area if it exists. Therefore we have $\sum_{j=1}^{q-1}\cI_{\Lambda}(r_{a,b})_{R_{1,j}}\ge 0$, and it is positive if there is no step-up at the leftmost descending arc in the braid area. This means $P_{a,b}$ contains the whole descending arc in the leftmost end of the braid area. Hence $P_{a,b}$ contains the most interior long overhead arc.
\end{proof}

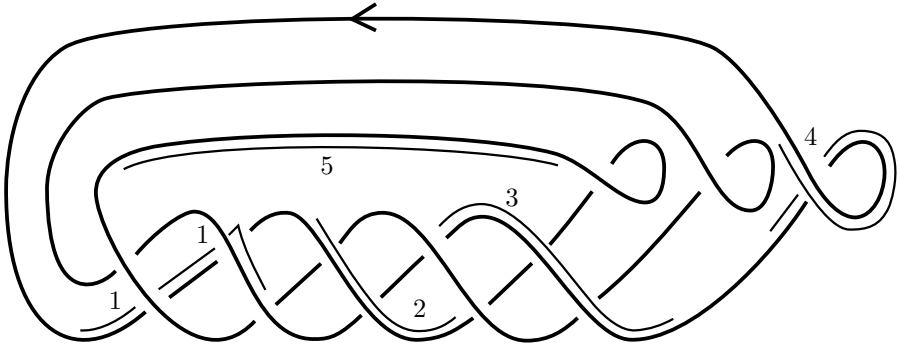
\begin{figure}[htb] {\small
\begin{overpic}
{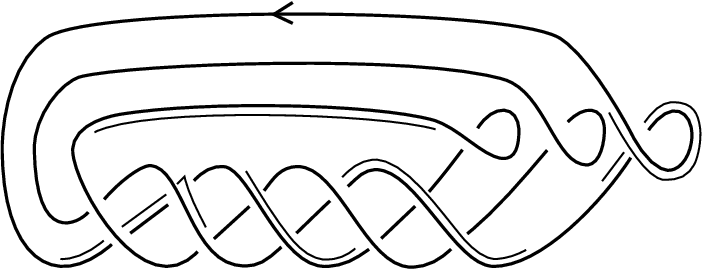}
\put(40, 13){$1$}
\put(73, 38){$1$}
\put(155, 10){$2$}
\put(190, 52){$3$}
\put(303, 75){$4$}
\put(120, 64){$5$}
\end{overpic}}
\caption{In the  push-out, 1 denotes a step-up,
2 denotes a horizontal piece, 3 denotes a complete descending arc, 4 denotes a piece near $\alpha_i$, and 5 denotes the most interior long overhead arc.  }
\label{figure:Torus knot 1 }
\end{figure}

By a complete descending arc, we mean an arc in $P_{a,b}$ within the braid area that descends from the top level to the bottom level and contains the two horizontal pieces at the ends. In other words, it does not contain the chord modification along the whole descending arc, see \Cref{figure:Torus knot 1 }.
\begin{lemma}\label{lemma:A2}
    If $P_{a,b}$ does not contain the most interior long overhead arc, then 
    $$\cI_{\Lambda}(r_{a,b})_{A_2}\ge 0.$$
    Moreover, the above quantity is positive if and only if $P_{a,b}$ contains a complete descending arc.
\end{lemma}
\begin{proof}
    Since $P_{a,b}$ does not contain the most interior long overhead arc,  any strand in $P_{a,b}$ must cover the region $A_2$ positively, i.e.\ each strand in  $P_{a,b}$ contributes $1$ to $\cI_{\Lambda}(r_{a,b})_{A_2}$. Each step-up in a strand within the braid area contributes one copy of $\mu$ to $\cI_{\Lambda}(r_{a,b})$, while the crossing at $\alpha_i$ contributes $-\mu$ to $\cI_{\Lambda}(r_{a,b})$. An occurrence of $\mu$ contributes $\frac{1-p}{(p-1)(q-1)}=\frac{-1}{q-1}$ to $\cI_{\Lambda}(r_{a,b})_{A_2}$. Each strand of $P_{a,b}$ including the crossing at $\alpha_i$ contributes at most $q-1$ copies of $\mu$. When the extremum is achieved,  the strand has a step-up at each descending arc of the braid area, or equivalently, the strand does not contain a complete descending arc. Furthermore, the total contribution to $\cI_{\Lambda}(r_{a,b})_{A_2}$ is $0$. When the strand has a complete descending arc,  the total contribution to $\cI_{\Lambda}(r_{a,b})_{A_2}$ is positive as we do not have enough $\mu$.
\end{proof}

\begin{proof}[Proof of \Cref{thm:main}]
    To solve \eqref{eqn:inequality}, by \Cref{lemma:sum}, we see that $x_{a,b}=0$ if $P_{a,b}$ contains the most interior long overhead arc. Then by \Cref{lemma:A2}, $x_{a,b}=0$ if $P_{a,b}$ does not contain the most interior long overhead arc and has no complete descending arc. As a consequence, we are left with $x_{a,b}$ such that $P_{a,b}$ contains no complete descending arc. First of all, $P_{a,b}$ must have only one strand otherwise it will contain a complete descending arc since $q>p$. Moreover, $\lfloor q/p \rfloor=1$ for otherwise, it will contain a complete descending arc. There is at most one such $P_{a,b}$. To have both the one-strand condition and the no complete descending arc condition, we must have $p=a+1,q=a$, contradiction. Therefore $x_{a,b}=0$ for all $a,b$. Then by \Cref{prop:no_curve} and \Cref{thm:curve}, we have $\partial_{CH}(q_{(\alpha_1)})=\pm 1$. Hence, the contact homology of $(\R^3_{\Lambda^+}, \xi_{\Lambda^+})$ vanishes, and so does the contact homology of $(S^3_{\Lambda^+}, \xi_{\Lambda^{+}})$ by \cite[Theorem 2.6(4)]{avdek2020combinatorial}.
    
    Since $tb(\Lambda)=pq-p-q=2g_{s}(\Lambda)-1$, where $g_{s}(\Lambda)$ is the slice genus of $\Lambda$,  according to the proof of \cite[Theorem 1.1]{LS}, $(S^3_{\Lambda^+}, \xi_{\Lambda^{+}})$ is tight. Moreover, the contact 3-manifold obtained by contact $1/k$ surgery along $\Lambda$ is tight.

    The contact $1/k$ surgery along $\Lambda$ is equivalent to $k$ contact $+1$
surgeries along $k$ Legendrian push-offs of $\Lambda$ \cite{DG}. So there exists a Liouville cobordism from the contact 3-manifold obtained by contact $1/k$ surgery along $\Lambda$ to  $(S^3_{\Lambda^+}, \xi_{\Lambda^+})$. By the Liouville functoriality of contact homology  \cite{P, BH}, the vanishing of the  contact homology of $(S^3_{\Lambda^+}, \xi_{\Lambda^+})$ implies the vanishing of the contact homology of the contact 3-manifold obtained by contact $1/k$ surgery along $\Lambda$.\end{proof}

\section{Some other rainbow closures of positive braids}
In this section, we consider some Legendrian knots which are Legendrian rainbow closures of positive braids of $p$ strands. The braids are generated by the generators $\sigma_1$, $\sigma_2$, $\cdots$, $\sigma_{p-1}$.

\begin{lemma}\label{lem:tight}
Suppose $\Lambda$ is the Legendrian rainbow closure of a positive braid whose Thurston-Bennequin invariant is not $-1$. Then the contact 3-manifold $(S^{3}_{1/k}(\Lambda), \xi_{1/k}(\Lambda))$ is tight for $k\in\N_+$.
\end{lemma}

\begin{proof}
Suppose the closed positive braid representation of $\Lambda$ has $p$ strands and consists of $w$ positive generators. From the Lagrangian projection, it follows that $tb(\Lambda)=w-p.$ On the other hand, there exists a Seifert surface of $\Lambda$ with Euler characteristic $p-w.$ Hence we have $$w-p=tb(\Lambda)\leq 2g_{s}(\Lambda)-1\leq 2g_{3}(\Lambda)-1\leq w-p,$$  where $g_{s}(\Lambda)$ is the slice genus of $\Lambda$, and $g_{3}(\Lambda)$ is the Seifert genus of $\Lambda$. Thus $\Lambda$ attains the maximal Thurston-Bennequin invariant, which is $tb(\Lambda)=2g_{s}(\Lambda)-1$. Moreover, since $tb(\Lambda)\neq-1$, it follows that $g_{s}(\Lambda)>0$. According to \cite{LS}, $(S^{3}_{1/k}(\Lambda), \xi_{1/k}(\Lambda))$ is tight for $k\in\N_+$.
\end{proof}

\begin{figure}[htb] {\small
\begin{overpic}
{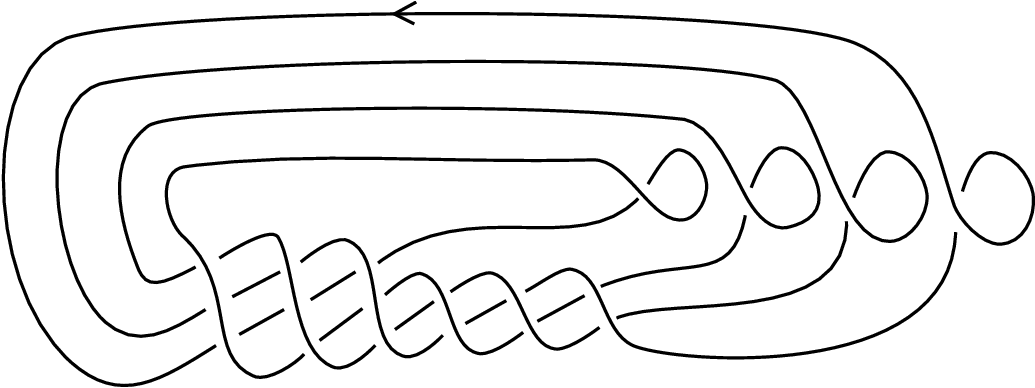}
\put(302, 105){$\alpha_1$}
\put(354, 105){$\alpha_2$}
\put(405, 105){$\alpha_3$}
\put(456, 105){$\alpha_4$}
\put(230, 90){$A_1$}
\put(250, 120){$A_2$}
\put(270, 140){$A_3$}
\put(290, 165){$A_4$}
\put(320, 90){$B_1$}
\put(370, 90){$B_2$}
\put(420, 90){$B_3$}
\put(470, 90){$B_4$}
\put(112, 55){$R_1$}
\put(150, 55){$R_2$}
\put(190, 40){$R_3$}
\put(222, 41){$R_4$}
\put(260, 43){$R_5$}
\end{overpic}}
\caption{ A Legendrian knot which is the rainbow closure of the braid $(\sigma_1\sigma_2\sigma_3)^{3}(\sigma_2\sigma_3)^{3}.$  }
\label{fig:twistedtorus }
\end{figure}

Let $\Lambda$ be a Legendrian knot which is the rainbow closure of the positive braid $$(\sigma_{1}\cdots\sigma_{p-2}\sigma_{p-1})^{q_1}(\sigma_{p-p_{2}+1}\cdots\sigma_{p-2}\sigma_{p-1})^{q_2}\cdots (\sigma_{p-p_{N}+1}\cdots\sigma_{p-2}\sigma_{p-1})^{q_N}.$$
Here $p_2<p$ and $p_i\ne p_{i+1}\le p$. The braid area can be divided into $N$ blocks which correspond to $(\sigma_{1}\cdots\sigma_{p-2}\sigma_{p-1})^{q_1}$, $(\sigma_{p-p_{2}+1}\cdots\sigma_{p-2}\sigma_{p-1})^{q_2}$, $\cdots,$ $(\sigma_{p-p_{N}+1}\cdots\sigma_{p-2}\sigma_{p-1})^{q_N},$ respectively. See Figure~\ref{fig:twistedtorus } for an example. 
Then we have
$$tb(\Lambda)=\sum_{i=1}^N (p_i-1)q_i-p.$$
 We label the regions above the braid area, as in the torus knot case, by $A_1,\ldots, A_p$ and $B_1,\ldots,B_p$, and crossings there by $\alpha_1,\ldots,\alpha_p$. The union of the first rows of $N$ blocks is ordered by $R_1,\ldots,R_{\sum_{i=1}^N q_{i}-1}$, as shown in Figure~\ref{fig:twistedtorus }. Due to the form of our braid, the regions of the first row of the first block are ordered by  $R_1,\ldots,R_{q_1-1}$, while the second block, as well as any block after, has an extra region before the first twist in the block. Therefore the regions of the first row of the $s$-th block are ordered by $R_{\sum_{i\le s-1} q_i},\ldots,R_{\sum_{i\le s} q_i-1}$ for $s\ge 2$. Note that the head of the $(s+1)$-th block, i.e.\ $R_{\sum_{i\le s} q_i}$, may cover a region extending to the left of the block, and hence stay above some of  $R_*$ before  $R_{\sum_{i\le s} q_i}$ if $p_{s+1}>p_s$. See Figure~\ref{figure:upperregion} for an example. 
 
 \begin{figure}[htb] {\small
\begin{overpic}
{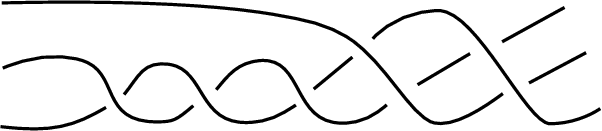}
\put(20, 15){$R_3$}
\put(70, 15){$R_4$}
\put(115, 15){$R_5$}
\put(130, 40){$R_6$}
\put(200, 40){$R_7$}
\put(260, 40){$R_8$}
\end{overpic}}
\caption{ A part of the braid area when $p_{3}>p_2$. }
\label{figure:upperregion}
\end{figure}
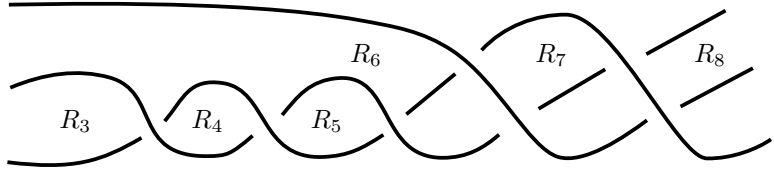
 
From \eqref{eqn:mu}, we have
 $$\cI_{\Lambda}(\mu)_{R_j}=-\frac{p-1}{(p-1)(q_1-1)+\sum_{i=2}^N(p_i-1)q_i}, \quad 1\le j \le q_1-1,$$
 and 
 $$\cI_{\Lambda}(\mu)_{R_j}=-\frac{p_i-1}{(p-1)(q_1-1)+\sum_{i=2}^N(p_i-1)q_i}, \quad \sum_{s=1}^{i}q_{s}\le j \le \sum_{s=1}^{i+1}q_{s}-1, \quad i\geq1.$$
As a consequence, we have
$$\sum_{i=1}^{\sum_{s=1}^N q_s-1}\cI_{\Lambda}(\mu)_{R_i}=-1.$$
Similar to the proof of \Cref{lemma:sum}, we have:
\begin{lemma}\label{lemma:top_row}
    Let $P$ be the push-out of a chord in the braid area. Then $$\sum_{i=1}^{\sum_{s=1}^N q_s-1}\cI_{\Lambda}(P)_{R_i}\ge 0. $$ Moreover, the equality holds if and only if $P$ does not contain the most interior long overhead arc.
\end{lemma}

\begin{proposition}\label{pro:q1>>0}
    If $q_1\gg 0$, then $(S^3_{\Lambda^+}, \xi_{\Lambda^+})$ has vanishing contact homology.
\end{proposition}
\begin{proof}
    We consider those push-outs $P$ that do not contain the most interior long overhead arc. Each strand contributes $1$ to $\cI_{\Lambda}(P)_{B_1}$. Note that 
    $$\cI_{\Lambda}(\mu)_{B_1}=-\frac{p-2}{(p-1)(q_1-1)+\sum_{i=2}^N(p_i-1)q_i}.$$
   Each strand has at most $\sum_{i=2}^Nq_i+q_1+2-\lfloor q_1/p\rfloor$ step-ups. This is a fairly loose inequality, and our primary concern is the leading coefficient of $q_1$. To see the bound, the free upper bound for step-ups is given by $\sum_{i=1}^N q_i$, meaning that each column of twist can contribute at most one step-up. If the Reeb chord is not in the first block of twists, then at least $\lfloor q_1/p\rfloor$ columns from the first block do not contribute. If the Reeb chord is in the first block of twists, then at least $\lfloor q_1/p\rfloor-2$ columns do not contribute. Along with the crossing at $\alpha_i$, we see that the contribution of $\mu$ from linking to $\cI_{\Lambda}(P)_{B_1}$ is at least 
    \begin{equation}\label{eqn:bound}
        -\frac{(p-2)(\sum_{i=2}^Nq_i+q_1+1-\lfloor q_1/p\rfloor)}{(p-1)(q_1-1)+\sum_{i=2}^N(p_i-1)q_i}.
    \end{equation}
    We have
    $$\lim_{q_1\to \infty} -\frac{(p-2)(\sum_{i=2}^Nq_i+q_1+1-\lfloor q_1/p\rfloor)}{(p-1)(q_1-1)+\sum_{i=2}^N(p_i-1)q_i}=-\frac{(p-2)(1-\frac{1}{p})}{p-1}>-1.$$
    Therefore, for $q_1\gg 1$, we have \eqref{eqn:bound} $> -1$.

    We will argue as in \Cref{thm:main} to show that $\partial_{CH}(q_{(\alpha_2)})=\pm 1$.  In light of \Cref{prop:positive}, it suffices to establish the analogue of Claim \ref{claim}, i.e.\
    $$\cI_{\Lambda}(\alpha_2)-\sum_{\substack{\text{Reeb chord }r_i \\ \text{ in the braid area}}} x_i \cI_{\Lambda}(r_i)-\sum_{1\le j \le p, j\ne 2} y_j \cI_{\Lambda}(\alpha_j)\ge 0,\quad x_i,y_j\ge 0,$$
    has only a trivial solution $x_i=y_j=0$ for any $i$ and $j$. Since $\cI_{\Lambda}(\alpha_i)=B_i$, by \Cref{lemma:top_row}, from looking at the sum of the coefficients of $R_i$, we have $x_i=0$ if $P_{r_i}$ contains the most interior long overhead arc. Then we can look at the coefficient of $B_1$, the discussion above shows that each strand not containing the most interior long overhead arc will contribute positively to $B_1$. Therefore we must have the rest $x_i=0$ and $y_1=0$. Finally, we can look at the coefficient of $B_3,\ldots, B_p$ to conclude that $y_3=\ldots=y_p=0$. Therefore $\partial_{CH}(q_{(\alpha_2)})$ only has constant terms and the constant term is $\pm 1$ by \Cref{thm:curve}.
\end{proof}

\begin{proposition}\label{pro:N=2}
    Let $\Lambda$ be a Legendrian knot which is the rainbow closure of the braid $$(\sigma_{1}\cdots\sigma_{p-2}\sigma_{p-1})^{q}(\sigma_{2}\cdots\sigma_{p-2}\sigma_{p-1})^{r}.$$
    Then $(S^3_{\Lambda^+}, \xi_{\Lambda^+})$ has vanishing contact homology if $q>1$.
\end{proposition}

\begin{proof}
    We consider those push-outs $P$ that do not contain the most interior long overhead arc. Each strand contributes $1$ to $\cI_{\Lambda}(P)_{B_1}$. Note that 
    $$\cI_{\Lambda}(\mu)_{B_1}=-\frac{p-2}{(p-1)(q-1)+(p-2)r}.$$
    Since each strand has at most $q+r$ step-ups, along with the crossing at $\alpha_i$, we see that the contribution of $\mu$ from linking to $\cI_{\Lambda}(P)_{B_1}$ is at least 
    $$-\frac{(p-2)(q+r-1)}{(p-1)(q-1)+(p-2)r}\ge -1+\frac{q-1}{(p-1)(q-1)+(p-2)r}.$$
    In particular,  we have 
    $$\cI_{\Lambda}(P)_{B_1}\ge \frac{q-1}{(p-1)(q-1)+(p-2)r}\ge 0.$$
    Since $\cI_{\Lambda}(\alpha_i)=B_i$, we can argue as in \Cref{pro:q1>>0} that $\partial_{CH}(q_{(\alpha_2)})=\pm 1$ when $q>1$. 
\end{proof}

\begin{proof}[Proof of Theorem~\ref{thm:2}]
Since $q_1>1$, $tb(\Lambda)=\sum_{i=1}^N (p_i-1)q_i-p_{1}\neq -1$. 
By the same argument as in the proof of Theorem~\ref{thm:main}, the theorem follows from Lemma~\ref{lem:tight}, Proposition~\ref{pro:q1>>0} and Proposition~\ref{pro:N=2}.
\end{proof} 

\section{ Chekanov's Legendrian  $5_2$ knots}
We consider  Chekanov's two  Legendrian $5_2$ knots  \cite[4.4]{Ch}.  Both Legendrian knots have Thurston-Bennequin invariants $1$ and rotation numbers $0$, yet they are not Legendrian isotopic. Note that the $5_2$ knot cannot be the closure of a positive braid. The Lagrangian projections of these two Legendrian knots are illustrated in Figure~\ref{fig:52}.
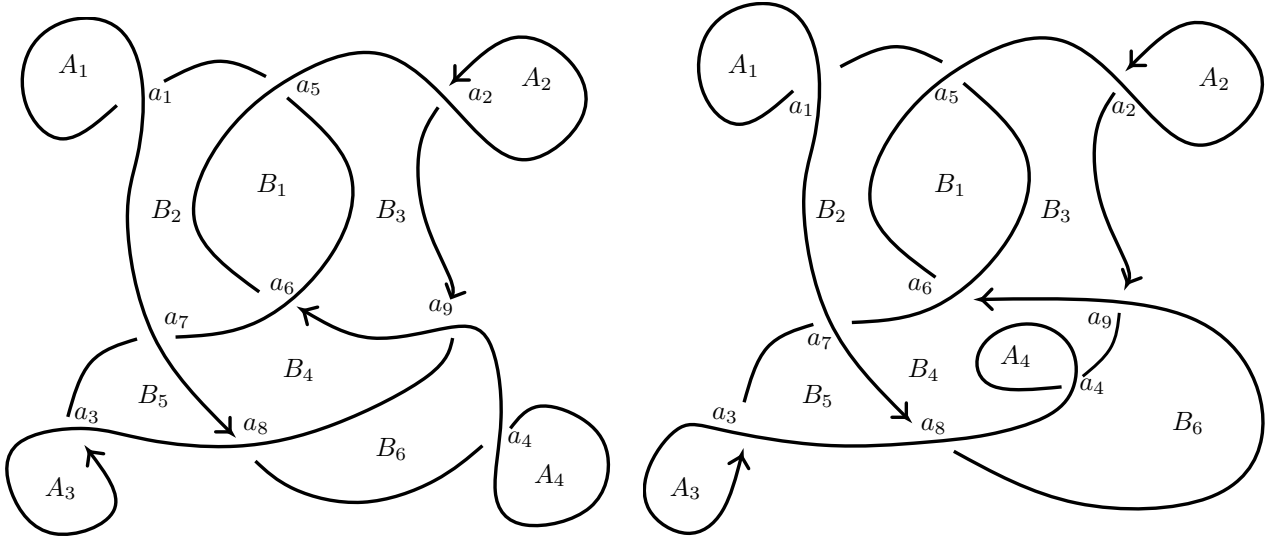
\begin{figure}[htb] {\small
\begin{overpic}[scale=0.5]
{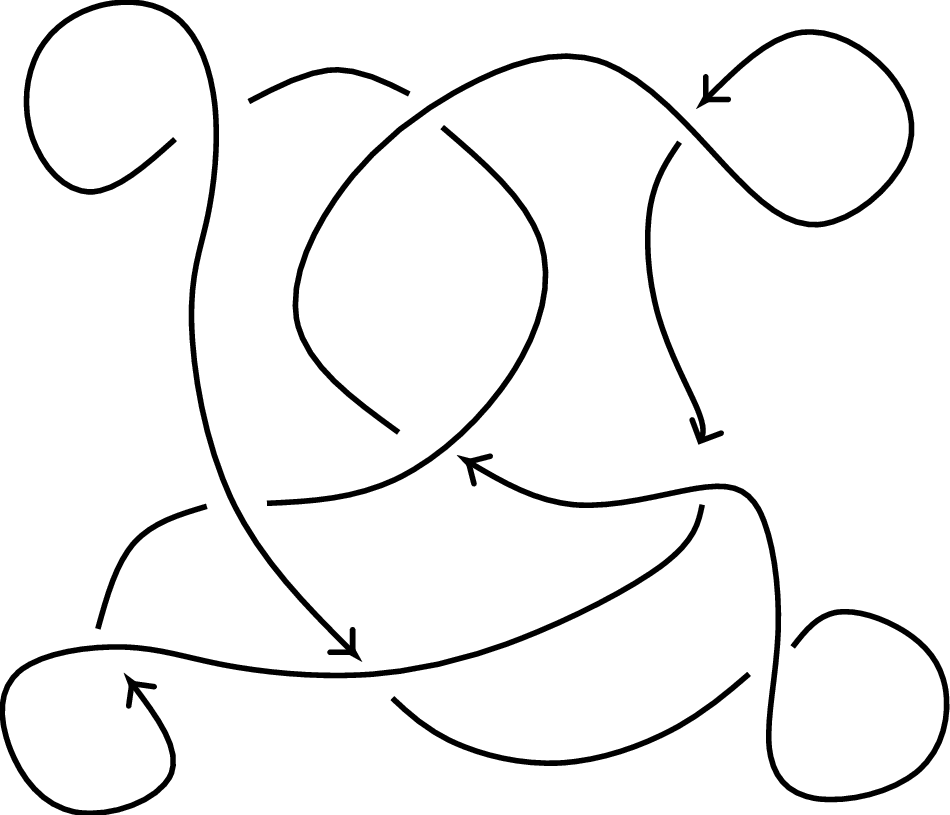}
\put (54,165) {$a_1$}
\put (175,165) {$a_2$}
\put (26,44) {$a_3$}
\put (190,35) {$a_4$}
\put (110,168) {$a_5$}
\put (100,93) {$a_6$}
\put (60,80) {$a_7$}
\put (90,40) {$a_8$}
\put (160,85) {$a_9$}
\put (20,175) {$A_1$}
\put (195,170) {$A_2$}
\put (15,15) {$A_3$}
\put (200,20) {$A_4$}
\put (95,130) {$B_1$}
\put (55,120) {$B_2$}
\put (140,120) {$B_3$}
\put (105,60) {$B_4$}
\put (50,50) {$B_5$}
\put (140,30) {$B_6$}
\end{overpic}
\quad 
\begin{overpic}[scale=0.5]
{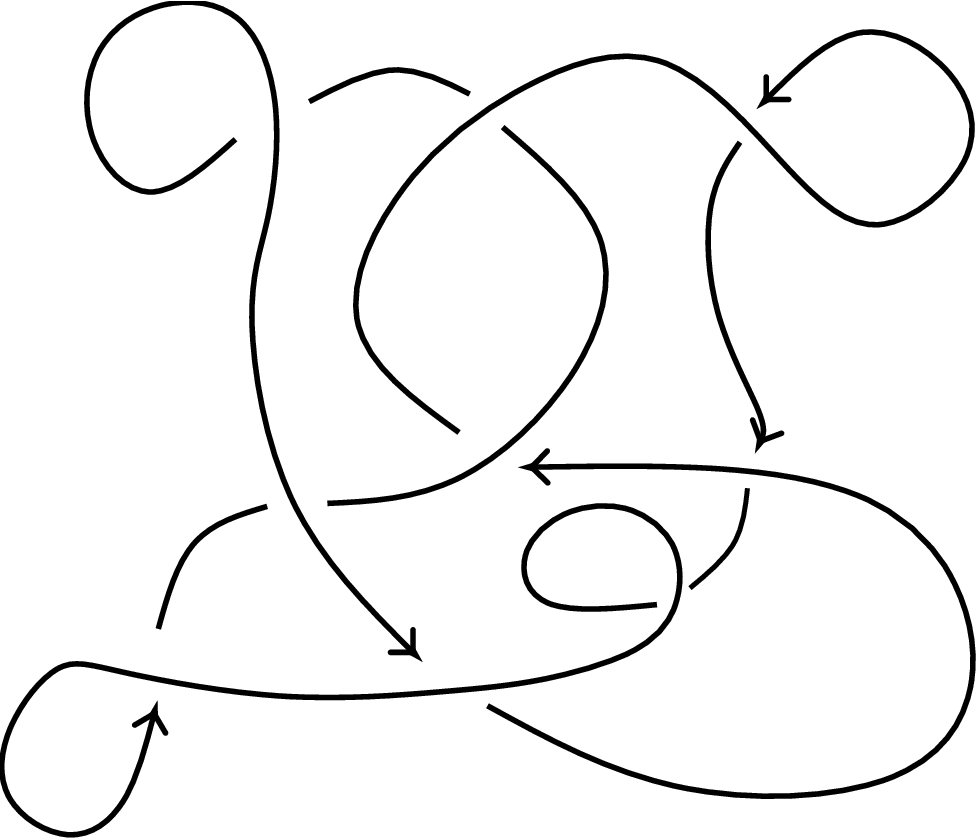}
\put (55,160) {$a_1$}
\put (177,160) {$a_2$}
\put (26,44) {$a_3$}
\put (165,55) {$a_4$}
\put (110,165) {$a_5$}
\put (100,93) {$a_6$}
\put (62,72) {$a_7$}
\put (105,40) {$a_8$}
\put (168,80) {$a_9$}
\put (32,175) {$A_1$}
\put (210,170) {$A_2$}
\put (10,15) {$A_3$}
\put (135,65) {$A_4$}
\put (110,130) {$B_1$}
\put (65,120) {$B_2$}
\put (150,120) {$B_3$}
\put (100,60) {$B_4$}
\put (60,50) {$B_5$}
\put (200,40) {$B_6$}
\end{overpic}}
\caption{ Chekanov's two  Legendrian $5_2$ knots.  }
\label{fig:52}
\end{figure}
One challenge for a general knot is that \Cref{prop:positive} fails in general, e.g.\ the $5_2$ knots above. Therefore the subalgebra of SFT degree $0$ may have more general generators and often infinitely many generators.
\begin{example}\label{ex:non-add}
    We label the crossings and regions of the two Lagrangian projections according to  Figure~\ref{fig:52}. For the left Legendrian $5_2$ knot, we compute that
    \begin{eqnarray*}
        \cI_{\Lambda}(\mu) & = & -\frac{1}{2}(-A_1+A_2+A_3-A_4+B_2-B_3-B_5+B_6);\\
        \cI_{\Lambda}(a_8) & = & -A_1+A_3+B_1+B_2-B_5+\cI_{\Lambda}(\mu);\\
        \cI_{\Lambda}(a_9) & = & A_2-B_1-B_3+\cI_{\Lambda}(\mu)\\
        \cI_{\Lambda}(a_8) + \cI_{\Lambda}(a_9) & = & A_4-B_6\\
        \cI_{\Lambda}(a_8a_9) & = & A_4-B_6+\cI_{\Lambda}(\mu).
    \end{eqnarray*}
    In particular, $\cI_{\Lambda}(a_8a_9)\ne \cI_{\Lambda}(a_8) + \cI_{\Lambda}(a_9)$. 
\end{example}
Because of the example above, computing $\cI_{\Lambda}(c)$ for single chords is insufficient to determine the intersection grading on the subalgebra of SFT degree $0$, which as mentioned above could have infinite generators. This poses serious challenges to execute Avdek's machinery in general.

\begin{proposition}
 Let $\Lambda$ be one of Chekanov's two  Legendrian $5_2$ knots in $(S^{3}, \xi_{std})$, then $(S^{3}_{1/k}(\Lambda), \xi_{1/k}(\Lambda))$ is algebraically overtwisted and tight  for $k\in \N_+$.
\end{proposition}
\begin{proof}
Since $tb(\Lambda)=1$ and $g_{s}(\Lambda)=1$, $tb(\Lambda)=2g_{s}(\Lambda)-1$. By \cite{LS}, $(S^{3}_{1/k}(\Lambda), \xi_{1/k}(\Lambda))$ is tight  for $k\in \N_+$. To show that $(S^{3}_{1/k}(\Lambda), \xi_{1/k}(\Lambda))$ is algebraically overtwisted, it suffices to show that $(S^3_{\Lambda^+},\xi_{\Lambda^+})$ is algebraically overtwisted.

We label the crossings and regions of the two Lagrangian projections according to  Figure~\ref{fig:52}. For the left Legendrian $5_2$ knot in Figure~\ref{fig:52}, we have the following non-trivial constraints for the contact action of Reeb chords.
\begin{eqnarray*}
\cA(a_2)-\cA(a_5)-\cA(a_6)-\cA(a_9)&=&Area(B_3)>0;\\
\cA(a_1)-\cA(a_5)-\cA(a_6)-\cA(a_7)&=&Area(B_2)>0;\\
\cA(a_4)-\cA(a_8)-\cA(a_9)&=&Area(B_6)>0;\\
\cA(a_3)-\cA(a_8)-\cA(a_7)&=&Area(B_5)>0.
\end{eqnarray*}
By arranging $B_6$ with small area $\epsilon$, we can achieve 
$$0\ll \cA(a_8)=\cA(a_9)<\cA(a_4)=2\cA(a_8)+\epsilon\ll \cA(a_7)\ll \cA(a_3)\ll \cA(a_5)\ll \cA(a_6)\ll \cA(a_2)\ll \cA(a_1).$$
Then by \Cref{prop:period}, $\partial_{CH}(q_{(a_4)})$ can have possibly nonzero coefficients only in $1, q_{(a_8)}$,  $q_{(a_9)}$,  $q_{(a_8^2)}$, $q_{(a_9^2)}$, $q_{(a_8a_9)}$. As we computed in \Cref{ex:non-add} that 
\begin{eqnarray*}
        \cI_{\Lambda}(a_8) & = &-\frac{1}{2}A_1-\frac{1}{2}A_2+\frac{1}{2}A_3+\frac{1}{2}A_4+B_1+\frac{1}{2}B_2+\frac{1}{2}B_3-\frac{1}{2}B_5-\frac{1}{2}B_6;\\
        \cI_{\Lambda}(a_9) & = & \frac{1}{2}A_1+\frac{1}{2}A_2-\frac{1}{2}A_3+\frac{1}{2}A_4-B_1-\frac{1}{2}B_2-\frac{1}{2}B_3+\frac{1}{2}B_5-\frac{1}{2}B_6;\\
        \cI_{\Lambda}(a_8^2) & = & 2\cI_{\Lambda}(a_8);\\
         \cI_{\Lambda}(a_9^2) & = & 2\cI_{\Lambda}(a_8);\\
        \cI_{\Lambda}(a_8a_9) & = & A_4-B_6+\cI_{\Lambda}(\mu).
\end{eqnarray*}
Each of them carries a positive coefficient for a term that is not $A_4$. Consequently, in light of \Cref{prop:no_curve} and \Cref{thm:curve}, we have  $\partial_{CH}(q_{(a_4)})=\pm 1$. Therefore, the contact $+1$ surgery along $\Lambda$ yields an algebraically overtwisted manifold.

For the right Legendrian $5_2$ knot in Figure~\ref{fig:52}, we have the following non-trivial constraints for the contact action of Reeb chords.
\begin{eqnarray*}
\cA(a_2)-\cA(a_5)-\cA(a_6)-\cA(a_9)&=&Area(B_3)>0;\\
\cA(a_1)-\cA(a_5)-\cA(a_6)-\cA(a_7)&=&Area(B_2)>0;\\
\cA(a_4)-\cA(a_8)-\cA(a_9)&=&Area(B_6)>0;\\
\cA(a_3)-\cA(a_8)-\cA(a_7)&=&Area(B_5)>0;\\
\cA(a_6)+\cA(a_7)+\cA(a_8)+\cA(a_9)-2\cA(a_4) & = & Area(B_4)>0.
\end{eqnarray*}
We can still arrange that 
$$0\ll \cA(a_8)=\cA(a_9)<\cA(a_4)=2\cA(a_8)+\epsilon\ll \cA(a_7)\ll \cA(a_3)\ll \cA(a_5)\ll \cA(a_6)\ll \cA(a_2)\ll \cA(a_1).$$
Therefore $\partial_{CH}(q_{(a_4)})$ can  have possibly nonzero coefficients only in $1, q_{(a_8)}$,  $q_{(a_9)}$,  $q_{(a_8^2)}$, $q_{(a_9^2)}$, $q_{(a_8a_9)}$ as before.
\begin{eqnarray*}
        \cI_{\Lambda}(a_8) & = &-\frac{1}{2}A_1-\frac{1}{2}A_2+\frac{1}{2}A_3+\frac{1}{2}A_4+B_1+\frac{1}{2}B_2+\frac{1}{2}B_3-\frac{1}{2}B_5-\frac{1}{2}B_6;\\
        \cI_{\Lambda}(a_9) & = & \frac{1}{2}A_1+\frac{1}{2}A_2-\frac{1}{2}A_3+\frac{1}{2}A_4-B_1-\frac{1}{2}B_2-\frac{1}{2}B_3+\frac{1}{2}B_5-\frac{1}{2}B_6;\\
        \cI_{\Lambda}(a_8^2) & = & 2\cI_{\Lambda}(a_8);\\
         \cI_{\Lambda}(a_9^2) & = & 2\cI_{\Lambda}(a_8);\\
        \cI_{\Lambda}(a_8a_9) & = & A_4-B_6+\cI_{\Lambda}(\mu).
\end{eqnarray*}
Each of them has a positive coefficient for a term that is not $A_4$, we can argue as before that the contact $+1$ surgery along $\Lambda$ yields an algebraically overtwisted manifold.
\end{proof}
\bibliographystyle{alpha} 
\bibliography{ref}
\Addresses
\end{document}